\documentclass[12pt,a4paper]{amsart}
\usepackage{amsfonts}
\numberwithin{equation}{section}

\theoremstyle{plain}
\newtheorem{Th}{Theorem}[section]
\newtheorem{Lemma}[Th]{Lemma}

 \theoremstyle{definition}

\newtheorem{?}[Th]{Problem}

\newcommand{\kong}[2]{ \equiv #1 \pmod{#2}}

\newcommand{\eq}[1]{\eqref{#1}}
     
\begin{document}

\title{Exact additive complements}

\author{Imre Z. Ruzsa}
\address{Alfr\'ed R\'enyi Institute of Mathematics\\
     Budapest, Pf. 127\\
     H-1364 Hungary
}
\email{ruzsa@renyi.hu}
 \thanks{Author was supported by ERC--AdG Grant No.321104  and
 Hungarian National Foundation for Scientific
 Research (OTKA), Grants No.109789 , 
 and NK104183.}

     \begin{abstract}
     Let $A,B$ be sets of positive integers such that $A+B$ contains all but finitely many positive integers.
S\'ark\"ozy and Szemer\'edi proved that if $ A(x)B(x)/x \to1$,
then $A(x)B(x)-x \to \infty$. 
Chen and Fang considerably improved S\'ark\"ozy and Szemer\'edi's bound.
We further improve their estimate and show by an example that our result is nearly best possible.

     \end{abstract}

     \maketitle

     \section{Introduction}

Two sets $A, B$  of positive integers are called \emph{additive complements} if their sumset $A+B$ 
contains all but finitely many positive integers. The counting functions of additive complements clearly satisfy
\begin{equation}\label{trivi}
A(x) B(x) \geq x-r, \end{equation}
where $r$ is the number of positive integers not represented as a sum. It is easy to construct sets, separating odd and even
places in a digital representation, for which equality holds for infinitely many values of $x$. These sets have the property that
 \[ \limsup A(x) B(x)/x > 1 . \]
Hanani asked whether this is always the case for infinite additive complements. This was answered by Danzer\cite{danzer64},
who first constructed infinite additive complements such that 
\begin{equation}\label{hyp} A(x)B(x)/x \to1. \end{equation}
We shall call such additive complements \emph{exact}. This property is less exotic than it seems; powers of a fixed integer
do have an exact complement, as do all sufficiently thin sets \cite{r96a, r01a}. 

Narkiewicz\cite{narkiewicz59} proved an important property of exact complements. He considered a wider class.

\begin{Th}[Narkiewicz's dichotomy]\label{nark}
Let $A,B$ be infitite sets of positive integers such that the number $r(x)$ of integers up to $x$ not contained in
their sumset $A+B$ satisfies $r(x)=o(x)$. Under condition \eq{hyp} we have
\begin{equation}\label{smallbig}
A(2x)/A(x) \to 1, \ B(2x)/B(x)\to 2, \end{equation}
or this holds with the roles of $A,B$ exchanged. If \eq{smallbig} holds, then for $\varepsilon>0$ and $x>x_0(\varepsilon)$ we have
\begin{equation}\label{smallbig2}
  A(x) < x^\varepsilon, \ B(x) > x^{1-\varepsilon} . \end{equation}
\end{Th}

 This shows that
polynomial sequences do not have an exact complement. The set of primes does not have either, for less obvious reasons \cite{r98f}.

For the sequel we will assume that \eq{smallbig} holds, that is, $A$ is small and
$B$ is large.

For exact complements S\'ark\"ozy and Szemer\'edi\cite{ssz94} proved that if \eq{hyp} holds,
then $A(x)B(x)-x \to \infty$. (While this paper actually appeared in 1994, the result was already announced in the 1966
edition of Halberstam and Roth's book Sequences\cite{sequences}.) They remark that their proof shows that
 \[ A(x)B(x)-x = o\big(A(x)\big) \]
cannot hold, and they conjecture that
 \[ A(x)B(x)-x = O\big(A(x)\big) \]
may be possible.
 
Chen and Fang\cite{chenfang} disproved this conjecture and considerably improved S\'ark\"ozy and Szemer\'edi's bound.
Their result shows that even
\begin{equation}\label{hatvany}
A(x)B(x)-x = O\big(A(x)^c\big) \end{equation}
cannot hold for any constant $c$.

The aim of this paper is to improve Chen and Fang's result and to show by means of an example that there
is precious little room for further improvement.

Write
 \[ a^*(x) = \max \{ a\in A, a \leq x .\}\]

\begin{Th}\label{estimate}
  Let $A, B$ be infinite  sets of positive integers such that the number $r(x)$ of integers up to $x$ not contained in
their sumset $A+B$ satisfies $r(x)=o(x)$. Suppose they satisfy \eq{hyp} and the notation corresponds to
\eq{smallbig}.  If $r(x)= o\bigl( a^*(x) \bigr)  $, then             we have
\begin{equation}\label{also}
    A(x) B(x) -x > \bigl(1-o(1)\bigr) \frac{a^*(x)}{A(x)} . \end{equation}
\end{Th}

The reason that this excludes \eq{hatvany} is that Narkiewicz's dichotomy \eq{smallbig2}   implies that
 \[ A(x) = A\bigl(a^*(x)\bigr) < a^*(x)^\varepsilon \]
hence $a^*(x)$ is larger than any power of $A(x)$. Chen and Fang's result, though stated in quite different terms,
is equivalent to the lower bound
 \[ \frac{2}{3} \sqrt{a^*(x)} .\]
The proof of Theorem \ref{estimate} is based on their argument, with some parts improved.

Clearly the bound in \eq{also} cannot be improved to $a^*(x)$, since for  $x\in A$  we have  $a^*(x)=x$, and this would contradict
\eq{hyp}. However, it is possible that such an improvement holds whenever  $a^*(x)$ is small compared to $x$.
It is also a natural question, also formulated by Chen and Fang, whether one can give an \emph{absolute} lower bound,
say $ A(x) B(x) -x> \log x$. We show this is not the case.

\begin{Th}\label{pelda}
  Let $\omega$ be a function tending to infinity arbitrarily slowly. There are additive complements satisfying \eq{hyp} such that
for infinitely many values of $x$ we have
\begin{equation}\label{felso}
 A(x) B(x) - x < \min \bigl(\omega(x), c a^*(x) \bigr) \end{equation}
with some constant $c$.
\end{Th}

\section{The lower estimate}

\begin{Lemma} \label{deltaszigma}
  Let $U,V$ be finite sets of integers. Put
   \[ \sigma(n) = \#\{ (u, v): u\in U, v\in V, u+v=n \}, \ \delta(n) = \#\{ (u, v): u\in U, v\in V, v-u=n \} . \]
   We have
    \[ \sum_{\sigma(n)>1} (\sigma(n)-1) \geq \frac{1}{|U|}  \sum_{\delta(n)>1} (\delta(n)-1)  . \]
\end{Lemma}

\begin{proof}
  We have
   \[ \sum \sigma(n) = \sum \delta(n) = |U| |V|,\]
\[ \sum \sigma(n)^2 = \sum \delta(n)^2 \]
by double-counting the quadruples satisfying $u+v=u'+v'$, which can be rearranged as $v-u'=v'-u$, and $\sigma(n)\leq |U|$ for all $n$.
Hence
 \[ \sum_{\delta(n)>1} (\delta(n)-1) \leq   \sum \bigl(\delta(n)^2 -\delta(n) \bigr) =  \sum \bigl(\sigma(n)^2 -\sigma(n) \bigr) \leq |U|  \sum_{\sigma(n)>1} (\sigma(n)-1).  \]
\end{proof}
This estimate can be doubled, as $ \delta(n)-1 \leq \bigl(\delta(n)^2 -\delta(n) \bigr)/2$ whenever ${\delta(n)>1}$, but we cannot utilize this improvement.

There are sets $U,V$ for which this estimate is correct up to a constant factor. It is likely that the sets for which we shall apply
this lemma are not of this kind, but I do not see any way to show this.

\begin{Lemma} \label{aatlag}
  Assume that the sets $A,B$ satisfy \eq{hyp} and \eq{smallbig}. Then
\begin{equation}\label{smallbiga}
A(cx)/A(x) \to 1 \end{equation}
uniformly in any range $c_1<c<c_2$ with $0<c_1<c_2$;
\begin{equation}\label{smallbigb}
\ B(cx)/B(x)\to c \end{equation}
uniformly in any range $c<c_2$ with $0<c_2$.
 Furthermore  
   \begin{equation}\label{aa}
  \sum_{a\in A, a\leq x} a = o(xA(x)) .  \end{equation}
\end{Lemma}
\begin{proof}
  For $c=2^k$ with a (positive or negative) integer $k$ the claim \eq{smallbiga} follows from an iterated application of
\eq{smallbig}. For general $c$ the claim for $A$ follows from the monotonicity of $A(x)$. For $B$ from \eq{hyp} we get
\eq{smallbigb} for the same range; the range can be extended down to 0 by the monotonicity of $B(x)$.

  To see \eq{aa} note that the sum with $a\leq \varepsilon x$ contributes at most $\varepsilon xA(x)$, and the sum with $a>\varepsilon x$ contributes at most
   \[ x \bigl(A(x)-A(\varepsilon x)\bigr) =o(xA(x))  \]
by \eq{smallbiga}.
\end{proof}

\begin{proof}[Proof of the Theorem.]
  Fix an integer $x$ and put $U=A\cap[1,x]$,  $V=B\cap[1,x]$. We use the notations $\sigma, \delta$ as in Lemma \ref{deltaszigma}.
We have
   \[ A(x)B(x)-x = |U| |V| -x = y + z - r,\]
   where
    \[ y =  \sum_{\sigma(n)>1} (\sigma(n)-1) \]
    counts the excess multiplicities,
     \[ z = \#\{n: n>x, n\in U+V \} \]
     counts the unnecessarily large sums, and $r=r(x)$ is the number of integers not in $A+B$.

     Let $t=a^*(x)$. Adding $t$ to any $b\in B, b>x-t$ we get a sum $>x$, so
      \[ z \geq B(x)-B(x-t) .\]
If $t \geq x/2$, we use only this and \eq{smallbig} with $c=(x-t)/x$ to conclude
 \[ z \geq \left(1- \frac{x-t}{x}  -o(1) \right)  B(x) \sim \frac{t}{x} B(x) \sim  \frac{t}{A(x)}.\]
(This argument works for $t>cx$ with any fixed $c>0$, but fails for very small $t$, which is the typical situation.)

Assume now $t<x/2$. We are going to estimate $y$. Put $V'=B\cap[1, x-t]$. We will consider the sets $V'+U$, $V'-U$, and
use $\sigma', \delta'$ to denote the corresponding representation functions. 

We have
 \[ \sum \delta'(n) = |U| |V'| = A(x) B(x-t) .\]
 As $U \subset [1,t]$ and $V'\subset[1, x-t]$, we have $V'-U \subset [1-t, x-t-1]$. We show that few sums lie in $[1-t, t]$.
Indeed, if $b-a \leq t$ with $a\in U, b\in V'$, then $b \leq a+t$, so for an $a\in A$ there are at most $B(a+t)$ possible choices of $b$,
This gives altogether
  \[ \sum_{a\in U} B(a+t) < (1+o(1))  \sum_{a\in U} \frac{a+t}{A(a+t)} \]
  by \eq{hyp}. As $A(a+t)=A(t)=A(x)= |U|$ in this range, the sum is equal to
   \[ t + \frac{1}{|U|}  \sum_{a\in U} a = (1+o(1)) t \]
by Lemma \ref{aatlag}. Hence
  \[ \sum_{a\in U} B(a+t) < (1+\varepsilon) t . \]

This means that at least $ A(x) B(x-t) - (1+\varepsilon) t$ pairs give a difference in the interval $[t+1, x-t-1]$, which contains
less than $x-2t$ integers. Consequently
 \[   \sum_{\delta'n)>1} (\delta'(n)-1)  > \bigl(A(x) B(x-t) - (1+\varepsilon) t \bigr) -(x-2t) =A(x) B(x-t) -x + (1-\varepsilon) t. \]
 We now apply Lemma \ref{deltaszigma} to the sets  $U, V'$ to conclude
  \[   \sum_{\sigma'(n)>1} (\sigma'(n)-1) \geq \frac{1}{|U|} \bigl( A(x) B(x-t) -x + (1-\varepsilon) t \bigr) = B(x-t)- \frac{x-(1-\varepsilon)t}{A(x)} . \]
Clearly $\sigma(n) \geq \sigma'(n)$ for all $n$, so 
   \[ y =  \sum_{\sigma(n)>1} (\sigma(n)-1) \geq   \sum_{\sigma'(n)>1} (\sigma'(n)-1). \]
   Adding the estimates we obtain
    \[ A(x)B(x)-x+r=y+z \geq B(x)- \frac{x-(1-\varepsilon)t}{A(x)} = \frac{A(x)B(x)-x}{A(x)} +\frac{(1-\varepsilon)t}{A(x)},\]
which can be rearranged as
\[ A(x)B(x)-x\geq  \frac{(1-\varepsilon)t}{A(x)-1} -  \frac{rA(x)}{A(x)-1}  .  \]

\end{proof}

\section{The construction}

We prove Theorem \ref{pelda}.

Take an increasing sequence $p_1, p_2, \ldots$ of primes such that $k^3 < p_k < (k+1)^3$, possibly with finitely many exceptions.
We shall construct a sequence of integers $u_k$ such that $u_k > ku_{k-1}$, $p_k|u_k$ and finite sets $A_i$ of integers such that
 \[ A_1 = \{ 1, 2, \ldots, p_1\}, \ A_k \subset (u_k, 2u_k)   \text{ for } k \geq 2,   \]
 \[ |A_1| = p_1, \ |A_k| = p_k - p_{k-1} \text{ for } k \geq 2, \]
 hence
  \[ |A_1 \cup A_2 \cup \ldots \cup A_k | = p_k , \]
  and the set $A_1 \cup A_2 \cup \ldots \cup A_k$ is a complete set of residues modulo $p_k$. One of the complements will be
   \[ A = \bigcup_{k=1}^\infty A_k . \]

   To specify the other set we put
    \[ B_k = \{ n: p_k | n, \ ku_k < n < (k+3) u_{k+1}  \} \]
and
 \[ B = \bigcup_{k=1}^\infty B_k . \]

First we prove that such sets  $A_k$
exist, provided the sequence $u_k$ increases sufficiently fast.

\begin{Lemma}
  There are integers $v_k$, depending only on the primes $p_j$, such that sets $A_k$ with the above described properties can
be found whenever $u_k > v_k $ for all $k$.
\end{Lemma}

\begin{proof}
Write
 \[ \delta = \prod_{i=k}^\infty \left(1- \frac{p_k-1}{p_j} \right) \]
 and choose $r$ so that
  \[ \sum_{i=r+1}^\infty \frac{1}{p_i} < \frac{\delta}{4p_k} . \]
The positivity of $\delta$ and the existence of $r$ follows from the convergence of the series $\sum 1/p_i$.
Write $q=p_k p_{k+1} \ldots p_r$. We show that  suitable sets can be found if $u_k > v_k = 2q/\delta$. 

We will construct the sets $A_k$ recursively. Given $A_1, \ldots, A_{k-1}$, a necessary condition for the existence of $A_k$
is that the elements of  $A_1 \cup A_2 \cup \ldots \cup A_{k-1}$ be all incongruent modulo $p_k$. Hence the property which we shall preserve
during the induction is:

``the elements of  $A_1 \cup A_2 \cup \ldots \cup A_k$ are all incongruent modulo $p_j$ for every $j \geq k$.''
We assume this holds for $k-1$ and we build $A_k = \{ a_1, a_2, \ldots, a_{ p_k - p_{k-1}}  \} $.

 Suppose 
$a_1, \ldots, a_{t-1}$ are already found. We want to find $a_t$ so that $m=p_k - p_{k-1} +t-1$ residue classes are
forbidden for each  $p_j$, $j \geq k$. In each interval of length $q$ there are
 \[ q  \prod_{i=k}^r \left(1- \frac{m}{p_j} \right)  > \delta q \]
integers which avoid the $m$ forbidden residue classes modulo all  $p_j$, $k \leq j \leq r$. In the interval $(u_k, 2u_k)$
this means at least $\delta u_k-q$ candidates.

Next we count the numbers in forbidden residue classes modulo  $p_j$, $j>r$.
The number of integers in a residue class $ a \pmod p$ in the interval $(u_k, 2u_k)$ is exactly
 \[ \left[ \frac{2u_k-a-1}{p}\right] - \left[ \frac{u_k-a}{p}\right] \leq \frac{2u_k}{p}, \]
 assuming that $p<2u_k$. We use this estimate for $p_j<2u_k$. This excludes less than
  \[ p_k  \sum_{i=r+1}^\infty \frac{2u_k}{p_i} < (\delta/2) u_k \]
integers.

 Finally, if  $p_j>2u_k$, then there are no new excluded integers. Indeed, the only integer satisfying
$ n \kong{a}{p_j}$ with some $a\in A_1 \cup A_2 \cup \ldots \cup A_{k-1} \cup \{ a_1, \ldots, a_{t-1} \} $ is $a$ itself, which was already excluded
(even several times) by previous congruences.

This leaves us  at least $(\delta/2) u_k-q$ integers to choose from, which is positive if $u_k > 2q/\delta$. 
\end{proof}

Now we show that $A, B$ are additive complements, then estimate $A(x)B(x)-x$.

To prove the first claim, take an arbitrary $n>3u_1$. It satisfies
 \[ (k+2)u_k < n \leq (k+3) u_{k+1} \]
 with some $k$. Select $a\in A$ so that
  \[ a\in A_1 \cup A_2 \cup \ldots \cup A_k, \ a \kong{n}{p_k} .\]
As $1 \leq a < 2u_k$, the integer $b=n-a$ satisfies $ku_k < b < (k+3) u_{k+1}$ and $p_k|b$, so $b\in B_k$.

Now we estimate $B(x)$ for a typical $x$. This number satisfies $ku_k<x \leq (k+1) u_{k+1}$ for some $k$. All blocks
$B_j, j>k$ lie above $x$. An initial segment of $B_k$ gives
 \[ B_k(x) \leq \frac{x-ku_k}{p_k} \]
 elements. To estimate the contribution of smaller blocks note that
  \[ |B_j| \leq \frac{(j+3) u_{j+1} - ju_j} {p_j},   \]
  hence
  \[ B(x) \leq  B_k(x) + |B_{k-1}| + |B_{k-2}| + \ldots + |B_1| \]
   \[ \leq \frac{x}{p_k} + \sum_{j=2}^k \left( \frac{j+2}{p_{j-1}} - \frac{j}{p_j} \right) u_j .\]
This estimate is not quite exact, since possibly only a segment of $B_{k-1}$ is contained in our interval, and the sets
$B_j$ are not disjoint; takint these into account would not substantially improve our result.

By our assumption about the rate of growth of the sequence $p_j$ the coefficient of $u_j$ in the above formula is $O(j^{-3})$,
that is,
 \[ B(x)  \leq \frac{x}{p_k} + c_1 \sum_{j=2}^k \frac{u_j}{j^3} < \frac{x}{p_k} + c_2 \frac{u_k}{k^3} \]
by our assumption about the rate of growth of the sequence $u_j$.

Since $A_{k+2}$ consists  already of elements $>u_{k+2} > (k+1) u_{k+1}$, we have $A(x) \leq p_{k+1}$, consequently
 \[ A(x) B(x) - x < \frac{p_{k+1}-p_k}{p_k} x +  c_2 \frac{u_kp_{k+1}}{k^3} = O(x/k) = o(x), \]
which shows that these sets are indeed exact complements.

For $x=u_{k+1}$ we have $A(x)=p_k$ and $u_k<a^*(x) < 2u_k$, so
\[ A(x) B(x) - x < c_2 \frac{u_kp_{k}}{k^3} < c_3 u_k < c_3a^*(x) , \]
and also
 \[  c_3 u_k < \omega(x) , \]
provided the sequence $u_j$ grows so fast that $\omega(u_{k+1}) > u_k$.
These estimates show the bound \eq{felso}. 

\section{Concluding remark}

All known constructions of exact complements use a variant of this approach, namely combining a complete set of residues 
modulo some integers $p_k$ (primes here, other sorts of integers in other papers, depending on the situation) and
multiples of these $p_k$ in an interval. The difficulty is that multiples of $p_k$ are needed for a time after the
appearance of the firts few multiples of $p_{k+1}$, which creates multiply represented sums. I see no way to eliminate or reduce
this effect, nor a way to improve the lower estimate which would then vindicate this overkill.

   \bibliographystyle{amsplain}
     \bibliography{cimek,cikkeim}



     \end{document}